\documentclass{daj}

\usepackage{amsmath, amsthm, amssymb}
\newtheorem{thm}{Theorem}

\newtheorem{lem}{Lemma}
\newtheorem{prop}{Proposition}
\newtheorem{cor}{Corollary}

\newtheorem{theorem}{Theorem}[section]

\theoremstyle{definition}

\theoremstyle{remark}
\newtheorem{remark}[theorem]{Remark}

\numberwithin{equation}{section}

\dajAUTHORdetails{%
  title = {A generalization of primitive sets and a conjecture of Erd\H{o}s}, %% please capitalize all significant words
  author = {Tsz Ho Chan, Jared Duker Lichtman, and Carl Pomerance},
  plaintextauthor = {Tsz Ho Chan, Jared Duker Lichtman, Carl Pomerance},
  plaintexttitle = {A generalization of primitive sets and a conjecture of Erdos}, %%  title without math or LaTeX
  runningtitle = {Erd\H{o}s 2-primitive set conjecture}, 
  runningauthor = {T. H. Chan, J. D. Lichtman, and C. Pomerance},
  keywords = {primitive set, primitive sequence},
}   %%% END \dajAUTHORdetails

\dajEDITORdetails{%
   year={2020},
   number={16},
   received={30 March 2020},   % received date: example: 7 January 2017
   published={23 September 2020},  % published date
   doi={10.19086/da.17290},       % XXX = number of paper, e.g. da006 for paper#6
}   %%% END \dajEDITORdetails

\begin{document}

\begin{frontmatter}[classification=text]
%% EDITOR: this will force the keywords to appear right after the Abstract.
%%   If the abstract is too long and would force the keywords off the
%%   front page, please comment out % [classification=text] above
%%   This way the keywords will be floated on the bottom of the first page
%%   even though the Abstract spills over to the next page.

%%% AUTHOR: Title goes here.  This line is optional.  You must use it
%%   if title has footnote attached or requires nontrivial typesetting,
%%   e.g., inclusion of linebreaks to force nice layout.
\title{A generalization of primitive sets and a conjecture of Erd\H{o}s} %% please capitalize all significant words

%%% AUTHOR:
%%% List all authors. If you wish, place grant acknowledgements in \thanks.
%%% In brackets include a short tag for each author.
\author[THC]{Tsz Ho Chan}
\author[JDL]{Jared Duker Lichtman}
\author[CP]{Carl Pomerance}

%%% AUTHOR: Abstract goes here
\begin{abstract}
A set of integers greater than 1 is {\it primitive} if no element divides another. Erd\H{o}s proved in 1935 that the sum of $1/(n \log n)$ for $n$ running over a primitive set $A$ is universally bounded over all choices for $A$. In 1988 he asked if this universal bound is attained by the set of prime numbers. 
We answer the Erd\H{o}s question in the affirmative for 2-primitive sets. Here a set is 2-{\it primitive} if no element divides the product of 2 other elements. 
\end{abstract}
\end{frontmatter}

%%% AUTHOR: body of paper starts here
\section{Introduction and Statement of results}

\bigskip

A set of integers greater than $1$ is called primitive if no element divides any other. Erd\H{o}s \cite{E35} showed that there is a constant $K$ such that for any primitive set $A$,
\[
f(A):=\sum_{n \in A} \frac{1}{n \log n} \le K.
\]
Further, in 1988 he asked if $f(A)$ is maximized by the primes $A = \mathbb{P}$.  This is now referred to as the Erd\H os conjecture for primitive sets:
\[
\hbox{For $A$ primitive, we have }~f(A)\le f(\mathbb{P})=  \sum_{p \in \mathbb{P}} \frac{1}{p \log p} =: C = 1.636616\cdots,
\]
where $\mathbb{P}$ is the set of prime numbers. By a simple argument, the Erd\H{o}s conjecture is equivalent to the assertion that $f(A) \le f(\mathcal P(A))$ for any primitive set $A$, where $\mathcal{P}(A)$ denotes the set of primes dividing some member of $A$.

Recently, the second and third authors \cite{LP} proved that
\begin{thm} \label{thm0}
For any primitive set $A$,
\[
f(A) < e^\gamma = 1.781072\cdots
\]
where $\gamma = 0.5772\cdots$ is the Euler-Mascheroni constant.
Further, if $A$ does not contain a multiple of $8$,  then
\[
f(A) \le f(\mathcal{P}(A)) + 2.37 \times 10^{-7}.
\]
\end{thm}

One can strengthen the notion of primitivity as follows. We say that a set $A$ of integers greater than 1 with $|A|\ge k+1$
is $k$-{\it primitive} if no element divides the product of $k$ distinct other elements. Note that $k$-primitive implies $j$-primitive for all $k\ge j\ge 1$.

In 1938, Erd\H{o}s \cite{E38} first studied the maximal cardinality of 2-primitive sets in an interval. 
The first author together with Gy\H{o}ri and S\'{a}rk\"{o}zy \cite{CGS} extended it to all $k$ and it was subsequently improved in \cite{C11} and \cite{PS}. While the full conjecture remains out of reach, we prove the Erd\H{o}s conjecture for 2-primitive sets (and hence $k$-primitive for all $k\ge2$).
\begin{thm} \label{thm1}
For any $2$-primitive set $A$,
\[
f(A) \le f(\mathcal{P}(A)).
\]
\end{thm}
An immediate consequence is the following
\begin{cor}
If $A$ is a primitive set with $f(A) > f(\mathcal{P}(A))$, then there exist three elements $a,b,c\in A$ with $a \mid  bc$.
\end{cor}

On the other hand, the primes are not optimal among primitive sets with respect to logarithmic density. Indeed, Erd\H{o}s, S\'{a}rk\"{o}zy, and Szemer\'{e}di \cite{ESS} obtained the best possible upper bound
\[
\mathop{\sum_{n \in A}}_{n \le x} \frac{1}{n} \le \Bigl( \frac{1}{\sqrt{2 \pi}} + o(1) \Bigr) \frac{\log x}{\sqrt{\log \log x}}
\]
for any primitive set $A$, while Erd\H{o}s \cite{E48} showed that
\[
\mathop{\sum_{n \in A'}}_{n \le x} \frac{1}{n} \ge \Bigl( \frac{1}{\sqrt{2 \pi}} + o(1) \Bigr) \frac{\log x}{\sqrt{\log \log x}}
\]
where $A'$ is the set of positive integers $a \le x$ with $\Omega(a) = [\log \log x]$. (Here,
$\Omega(a)$ is the number of prime factors of $a$, counted with multiplicity.)
By contrast, the primes satisfy
\[
\sum_{p \le x} \frac{1}{p} = \log \log x + O(1).
\]

Nevertheless, one may wonder if the primes still maximize the logarithmic density among 2-primitive sets. Indeed, we prove this to be the case.
\begin{prop} \label{prop1}
For all $x\ge2$ and any $2$-primitive set $A$,
\[
\mathop{\sum_{n \in A}}_{n \le x} \frac{1}{n} \le \mathop{\sum_{p \in \mathcal{P}(A)}}_{p\le x} \frac{1}{p}
\]
\end{prop}
We use this to deduce Theorem \ref{thm1}.

\begin{proof}[Proof of Theorem \ref{thm1} given Proposition \ref{prop1}]

By Proposition \ref{prop1}, we have $F(x)\ge0$ for all $x\ge2$, where
\[
F(x) := \mathop{\sum_{p \in \mathcal{P}(A)}}_{p \le x} \frac{1}{p} - \mathop{\sum_{n \in A}}_{n \le x} \frac{1}{n}
\]
Then by partial summation,
\[
\mathop{\sum_{p \in \mathcal{P}(A)}}_{p \le x} \frac{1}{p \log p} - \mathop{\sum_{n \in A}}_{n \le x} \frac{1}{n \log n} = 
\frac{F(x)}{\log x} + \int_{2^-}^{x} \frac{F(u)}{u \log^2 u} du \ge 0.
\]
Hence taking $x\to\infty$ gives $f(\mathcal P(A)) \ge f(A)$ as desired.
\end{proof}

\bigskip

In light of Proposition \ref{prop1}, it is natural to ask if there exists an exponent $\lambda<1$ for which 
\begin{equation} \label{bankmartin}
\sum_{\substack{n \in A\\n\le x}} \frac{1}{n^\lambda} \le \sum_{\substack{p \in \mathcal{P}(A)\\p\le x}} \frac{1}{p^\lambda}
\end{equation}
holds for all 2-primitive $A$, $x\ge2$. Banks and Martin \cite{BM} settled the question in the setting of 1-primitive sets, proving \eqref{bankmartin} holds for all primitive $A$ if and only if
\begin{align*}
\lambda \ge \tau_1:=1.1403659\cdots,
\end{align*}
where $t=\tau_1$ is the unique real solution to the equation
\begin{align*}
\sum_\mathbb{P} p^{-t} = 1+\Big(1-\sum_\mathbb{P} p^{-2t}\Big)^{1/2}.
\end{align*}
The fact that $\tau_1$ is markedly larger than 1 gives some indication as to why the full Erd\H{o}s conjecture remains open.

In the setting of 2-primitive sets, we extend the range of valid exponents $\lambda$.
\begin{thm} \label{thm5}
For any $\lambda \ge 0.7983$, $x\ge2$, and any $2$-primitive set $A$,
\begin{equation}\label{eq:lambdaprop}
\sum_{\substack{n \in A\\n\le x}} \frac{1}{n^\lambda} \le \sum_{\substack{p \in \mathcal{P}(A)\\p\le x}} \frac{1}{p^\lambda}
\end{equation}
\end{thm}
We remark it suffices to verify Theorem \ref{thm5} with $\lambda=0.7983$.  Indeed, suppose that  $F_\lambda(x)\ge0$ for all $x\ge2$, where 
\[
F_t(x) = \sum_{\substack{p\in \mathcal P(A)\\ p\le x}}p^{-t} - \sum_{\substack{n\in A\\ n\le x}} n^{-t}.
\]
Then, by partial summation, for any $t>\lambda$,
\begin{align*}
F_t(x) = x^{\lambda - t}F_{\lambda}(x) + (t-\lambda)\int_2^x u^{\lambda-t-1} F_\lambda(u) \;du \ge 0.
\end{align*}
Hence we may define the {\bf critical exponent} $\tau_2$ for 2-primitive sets, as the infimum over all $\lambda$ for which \eqref{eq:lambdaprop} holds.
Thus, Theorem \ref{thm5} implies that $\tau_2\le0.7983$.

We also note that Theorem \ref{thm5} with $\lambda=1$ gives us Proposition \ref{prop1}.
However, Theorem \ref{thm5} does not hold for every positive value of $\lambda$.  Indeed, 
in \cite{E38}, Erd\H{o}s showed that there is a $2$-primitive set $A$ in $[1,x]$ of cardinality $\pi(x) - \pi(x^{1/3}) + c x^{2/3} / (\log x)^2$. It consists of primes in $(x^{1/3}, x]$ and a subset of $\{ p_1 p_2 p_3 : p_i \text{ are  primes } \le x^{1/3} \}$ where the triples $\{p_1, p_2, p_3\}$ form a Steiner triple system. Thus, by the prime number theorem,
\[
\sum_{n \in A} \frac{1}{n^\lambda} \ge \sum_{x^{1/3} < p \le x} \frac{1}{p^\lambda} + \frac{c x^{2/3}}{(\log x)^2} \frac{1}{x^\lambda} > \sum_{p \le x} \frac{1}{p^\lambda}
\]
when $\lambda < 0.5$ and $x$ is sufficiently large. Hence the above argument and Theorem \ref{thm5} together imply that the critical exponent lies in the interval
\begin{align}
0.5 \le \tau_2 \le 0.7983.
\end{align}

In a sequel paper, we shall address the question of critical exponents for $k$-primitive sets, with $k\ge3$.

%--------------------------------------------------------------------------------------------
\section{Combinatorial Lemmas}

Before proving Theorem \ref{thm5}, we need lemmas in counting the maximal number of elements in a $k$-primitive set.

We first recall the following famous result due to Erd\H os and Szekeres \cite{ES},
whose proof we provide for completeness.
\begin{lem}[Erd\H os--Szekeres]\label{lemES}
A sequence of $(r-1)(s-1)+1$ real numbers has either a monotonic nondecreasing subsequence of length $r$
or a monotonic nonincreasing subsequence of length $s$.
\end{lem}
\begin{proof}
Say the sequence is $a_1,a_2,\dots,a_n$, where $n=(r-1)(s-1)+1$.
For each $a_i$ consider the ordered pair $(b_i,c_i)$, where $b_i$ is the length of the
longest nondecreasing subsequence ending at $a_i$ and $c_i$ is the length of the longest nonincreasing
subsequence ending at $a_i$.  Then no two pairs $(b_i,c_i)$ and $(b_j,c_j)$ can be equal, 
so for at least one choice of $i$ we have $b_i\ge r$ or $c_i\ge s$.
\end{proof}

We next bound the size of a $k$-primitive set based on the number of prime factors used to generate its elements.

\begin{lem} \label{lem2}
For $k\ge2$, suppose $A$ is a $k$-primitive set and $T\subset A$ with $|\mathcal{P}(T)| = n$. If $n\le k$, then $|T|\le n$. If $n=k+1$, then $|T|\le n+1$. 
Further, for $k=2,\, n=4$ we have $|T|\le 19$.
\end{lem}
\begin{proof} 
We may assume that $|T|\ge n$.
Let $\mathcal P(T) = \{q_1,\ldots,q_n\}$ and write each $t=\prod_i q_i^{e_i}\in T$ as an exponent vector $\vec{v} = (e_1,\ldots,e_n)$. Define the notation $\vec{v}\ge0$ if $e_i\ge0$ for all $i$, and define $\vec{v}\le\vec{w}$ if
 $\vec{w}-\vec{v}\ge0$.
Take $\vec{v}_1$ with maximal entry $e_1$ among $T$.  Then take
$\vec{v}_2$ with maximal $e_2$ among the remaining vectors, and similarly define
 $\vec{v}_3,\dots,\vec{v}_n$.  Thus, the chosen vectors are distinct.

Case $n\le k$: If $|T|\ge n+1$ then $T$ has some vector $\vec{v}\neq \vec{v}_i$ for all $i$. But then $\vec{v} \le \vec{v}_1+\cdots+\vec{v}_n$.  This implies that $T$, and hence $A$, is not $n$-primitive, and since $n\le k$, it
implies that $A$ is not $k$-primitive, a contradiction.  Hence we cannot have $|T|\ge n+1$ when $n\le k$.

Case $n=k+1$: If $|T|\ge n+2$ then $T$ has vectors $\vec{w}_1\ne \vec{w}_2$
with $\vec{w}_j\notin\{\vec{v}_1,\dots \vec{v}_n\}$ for $j=1,2$. Write $\vec{w}_j = (f_1^{(j)},\ldots,f_n^{(j)})$. By the pigeonhole principle, we may assume
$$f_i^{(1)} \le f_i^{(2)}$$
for at least $n/2$ values of $i$, say $i=1,\ldots, \lceil n/2\rceil$. Thus, we deduce
\begin{align*}
\vec{w}_1 \le \vec{w}_2 + \vec{v}_{\lceil n/2\rceil+1}+\cdots+\vec{v}_n
\end{align*}
 contradicting $T$ as $k$-primitive, since $1 + \lfloor n/2 \rfloor=1+\lfloor(k+1)/2\rfloor \le k$.
 
Now say $k=2$, $n = 4$.
Suppose there are $20$ members in $T$ with corresponding vectors
\[
\vec{w}_i := (e_{i,1}, e_{i,2}, e_{i,3}, e_{i_4}) \text{ for } 1 \le i \le 20.
\]
Since $A$ is 2-primitive, so is $T$.
Without loss of generality, say $\vec{w}_{18}$ has maximal first coordinate, $\vec{w}_{19}\ne \vec{w}_{18}$
has maximal second coordinate among the remaining 19 vectors, and 
$\vec{w}_{20}$ has maximal third coordinate among the remaining 18 vectors with
$\vec{w}_{20}\ne\vec{w}_{18},\vec{w}_{19}$.
Arrange the remaining $17$ vectors in ascending order of their first coordinate (i.e.,  $e_{1,1} \le e_{2,1} \le ... \le e_{17,1}$). By Lemma \ref{lemES}, there is a monotonic sequence of length $5$ among the $e_{i,2}$'s. Without loss of generality, say $e_{1,2}$, $e_{2,2}$, $e_{3,2}$, $e_{4,2}$, $e_{5,2}$ form such a sequence.

Case 1: $e_{1,2} \le e_{2,2} \le e_{3,2} \le e_{4,2} \le e_{5,2}$. Consider the numbers $e_{i,3}$ for $i=1,\dots,5$.
By Lemma \ref{lemES}, there is a monotonic sequence 
of length $3$ among the $e_{i,3}$'s, without loss of generality, say it is $e_{1,3},e_{2,3},e_{3,3}$.
If $e_{1,3} \le e_{2,3} \le e_{3,3}$, this forces $e_{2,4} > e_{1,4} + e_{3,4}$ for otherwise $\vec{w}_2\le\vec{w}_1+\vec{w}_3$, contradicting $T$ being 2-primitive.
 But this implies that $ \vec{w}_1 \le \vec{w}_2$  which contradicts $T$ being primitive. Hence, we must have $e_{1,3} \ge e_{2,3} \ge e_{3,3}$. Again, this forces $e_{2,4} > e_{1,4} + e_{3,4}$, which in turn implies that
 $\vec{w}_1\le\vec{w}_2+\vec{w}_{20}$, again a contradiction.
 
Case 2: $e_{1,2} \ge e_{2,2} \ge e_{3,2} \ge e_{4,2} \ge e_{5,2}$. By Lemma \ref{lemES}, there is a monotonic sequence of length $3$ among the $e_{i,3}$'s, without loss of generality, say it is $e_{1,3},e_{2,3},e_{3,3}$.  If 
$e_{1,3}\le e_{2,3}\le e_{3,3}$, then again this forces $e_{2,4}>e_{1,4}+e_{3,4}$.
But then $\vec{w}_1\le\vec{w}_2+\vec{w}_{19}$.
 Hence, we must have $e_{1,3} \ge e_{2,3} \ge e_{3,3}$. This forces $e_{2,4} > e_{1,4} + e_{3,4} $. But then 
 $\vec{w}_3\le\vec{w}_2+\vec{w}_{18}$, again a contradiction.

Therefore, there can be at most $19$ members in $T$.
\end{proof}

\begin{remark}
It is not clear if the number ``19" in Lemma \ref{lem2} is optimal.
We will not need it here, but by similar methods one can prove that if
$T$ is a 2-primitive set of positive integers with $|\mathcal{P}(T)| = n\ge3$,
then $|T|\le 9^{2^{n-3}}$.
\end{remark}

\section{Proof of Theorem \ref{thm5}}
Let $A\subset(1,x]$ be a 2-primitive set. 
Let $0.79 \le \lambda < 1$ be a parameter to be defined later. First, we partition $A$ into primes $S$ and composites $T$. Note $S$ and $\mathcal{P}(T)$ are disjoint since $A$ is primitive. For a prime $p$, define
\[
T_p := \{ t \in T : p \mid  t \}.
\]
If some prime $p \in \mathcal{P}(T)$ satisfies
\begin{equation} \label{d1}
\sum_{t \in T_p} \frac{1}{t^\lambda} \le \frac{1}{p^\lambda},
\end{equation}
then we replace the members of $T_p$ with the prime $p$ (i.e., redefine $A=(T\backslash T_p)\cup \{ p \}$). This would make $\sum_{T_p} t^{-\lambda}$ at least as big while keeping $A$ 2-primitive. Repeat the process with each prime $p \in \mathcal{P}(T)$ until no such prime satisfies \eqref{d1}. If $T = \emptyset$ after doing this, then $A=S$ consists of primes so Proposition \ref{prop1} follows. Otherwise $T \neq \emptyset$, so we may assume
\begin{equation} \label{lambda1}
\sum_{t \in T_p} \frac{1}{t^\lambda} > \frac{1}{p^\lambda} \qquad \text{ for all}\quad p \in \mathcal{P}(T).
\end{equation}
Consider the set
\begin{align}\label{eq:D}
    D := \{t / p : t \in T ,~ p \mid  t \}
\end{align}

We record some useful properties of $T$ and $D$.
\begin{lem}\label{lem}
Let $T$ be a $2$-primitive set for which \eqref{lambda1} holds and let $D$ be as in \eqref{eq:D}.
\begin{enumerate}
\item[(i)] For each $p\in{\mathcal P}(T)$,  $T_p$ has at least $3$ elements.
\item[(ii)] The map sending ordered pairs $(t, p)$ with $t\in T$ and $p\mid t$ to $t/p\in D$ is injective.
\item[(iii)] Each $t\in T$ has at least $3$ prime factors (counted with multiplicity).
\item[(iv)] $D$ is a primitive set of composite numbers.
\end{enumerate}
\end{lem}
\begin{proof}
(i) 
For $p\in{\mathcal P}(T)$, \eqref{lambda1} implies that 
\[
\sum_{t\in T_p}\frac1{(t/p)^\lambda}>1>2^{-0.79}+3^{-0.79},
\]
Thus (i) follows, since $t/p\in \mathbb{Z}_{>1}$ for all $t\in T_p$.

(ii)  If not, then $t_1 / p_1 = t_2 / p_2$ for some $t_1, t_2$, $p_1 \mid t_1$, and $p_2 \mid  t_2$.  If $t_1 \neq t_2$, by (i) there exists some $p_1 k \in T_{p_1}$ other than $t_1,t_2$. But then $t_1 = (t_1 / p_1) p_1 = (t_2 / p_2) p_1 \mid t_2 (p_1 k)$, which contradicts $T$ as 2-primitive. Hence $t_1 = t_2$, which forces $p_1 = p_2$.

(iii) If not, say $t = p q$. Since $T_p,T_q$ each have at least 3 elements, there are some $p m$ and $q n$ other than $t\in T$. But then, $t = p q \mid  (p m)(q n)$ which contradicts $T$ as 2-primitive.  (This argument holds whether or not $p\ne q$.)

(iv) If not, then $(t / p) \mid  (t_1 / p_1)$ for some $t, t_1 \in T$, $p \mid  t$, $p_1 \mid  t_1$,
and $t/p\ne t_1/p_1$. If $p_1 = p$, then $t \mid  t_1$ which contradicts $T$ as primitive. And if $p_1 \neq p$, then there is some $p l \in T_p$ other than $t$ and $t_1$. This implies $t \mid  t_1\cdot p l$, and since
$t\ne t_1$ (otherwise $p=p_1$), we have a contradiction to $T$ being 2-primitive. Thus $D$ is primitive, and also composite by (iii).
\end{proof}

For Theorem \ref{thm5}, we must show
\begin{equation} \label{transform1}
 \sum_{t \in T} \frac{1}{t^\lambda} - \sum_{p \in \mathcal{P}(T)} \frac{1}{p^\lambda} < 0.
\end{equation}
Suppose $\mathcal{P}(T)$ consists of primes $q_1 < q_2 < \cdots < q_r$. Let $2 = p_1 < p_2 < \cdots < p_r$ be the first $r$ primes in $\mathbb{P}$. We are going to modify the set $T$ by the following process. First, if each $q_i=p_i$,
we let $T$ stand as it is.  Otherwise, let $i$ be the smallest index such that $q_i > p_i$.  Then $q_j=p_j$
for all $j<i$ and we have $p_i\nmid t$ for all $t\in T$.  Then replace each $t\in T_{q_i}$ with $p_i/q_i\cdot t$.
This keeps $T$ as 2-primitive, and by \eqref{lambda1},
\[
0 \ < \ \sum_{t\in T_{q_i}}\frac1{t^\lambda} -\frac1{q_i^\lambda} \ < \ \frac{q_i^\lambda}{p_i^\lambda}\Biggl(\sum_{t\in T_{q_i}}\frac1{t^\lambda}-\frac1{q_i^\lambda}\Biggr)
=\sum_{t\in T_{q_i}}\frac1{(p_i/q_i\cdot t)^\lambda}-\frac1{p_i^\lambda}.
\]
So replacing each $t\in T_{q_i}$ with $p_i/q_i\cdot t$ preserves \eqref{lambda1}.  We repeat this
process for each $i$ with $q_i>p_i$ and in the end we have $\mathcal{P}(T)=\{p_1,p_2,\dots,p_r\}$.
By showing \eqref{transform1} for this $T$ it would follow that \eqref{lambda1} fails for some $p_i$,
and this contradiction would prove the theorem.

We have reduced Theorem \ref{thm5} to the following.
\begin{theorem}
\label{thm5'}
Suppose $\lambda\ge 0.7983$ and
$T$ is a $2$-primitive set of composite numbers satisfying \eqref{lambda1}
with $\mathcal{P}(T)=\mathbb{P}\cap(1,Y]$ for some $Y$.
Then 
\begin{equation} \label{transform2}
\sum_{t \in T} \frac{1}{t^\lambda} - \sum_{p \le Y} \frac{1}{p^\lambda} < 0.
\end{equation}
\end{theorem}

Our goal now is to prove Theorem \ref{thm5'}.
For a parameter $0 < \theta < 1$ to be chosen later, we define $\lambda$ as
\begin{align}
\lambda = \tau(1 - \theta), \hbox{ where }\tau = 1.140366.
\end{align}

First consider those $t \in T$ with greatest prime factor $P(t) \ge t^\theta$. Then $t^{1- \theta} \ge t / P(t)$ and so $t^{-\lambda} \le (t/P(t))^{-\lambda / (1 - \theta)} = (t/P(t))^{-\tau}$. Hence
\begin{equation} \label{largeprime}
\mathop{\sum_{t \in T}}_{P(t) \ge t^\theta} t^{-\lambda} \le \mathop{\sum_{t \in T}}_{P(t) \ge t^\theta} \Big(\frac{t}{P(t)}\Big)^{-\tau} \le \sum_{p \le Y} p^{-\tau}
\end{equation}
by \eqref{bankmartin}, since $\{ t/P(t):t\in T \} \subset D$ is primitive by part (iii) of Lemma \ref{lem}.

\bigskip

For a positive integer $t$, we consider the following unique factorization
\[
t = m(t) M(t)
\]
into positive integers $m(t) \le M(t)$ with ratio $M(t) / m(t)$ minimal. 
Let
\[
{\mathcal M}(T)=\{m(t):t\in T\}\cup\{M(t):t\in T\}.
\]
We need two lemmas.

\begin{lem} \label{lem1}
For any $2$-primitive set $T$, consider the graph on the integers with edges $\{t,m(t)\}$ and $\{t,M(t)\}$
for $t\in T$, where if $m(t)=M(t)$, there is just one edge containing $t$. 
This graph contains a matching from $T$ into ${\mathcal M}(T)$.
\end{lem}
\begin{proof}
Let $t\in T$.    
If $m(t)\notin\{m(t'),M(t')\}$ for all other $t'\in T$, then we can match $t$ with $m(t)$.
So assume $m(t)\in\{m(t'),M(t')\}$ for some other $t'\in T$.  
Then $M(t)\notin\{m(t''),M(t'')\}$
for all $t''\in T$ with $t''\ne t,t'$, since otherwise  $t\mid t't''$,
contradicting $T$ being 2-primitive.  

If $m(t)<M(t)$, then 2-primitive implies $M(t)\notin\{m(t'),M(t')\}$ so we can match $t$ with $M(t)$. 

Otherwise $m(t)=M(t)$, which means $t=m(t)^2$. Then $t'\neq t$ forces $m(t')<M(t')$, so we make define $m' = t'/m(t)$ (that is $m'$ is the singleton in $\{m(t'),M(t')\}\setminus\{m(t)\}$).
We would like to match $t'$ with $m'$ instead of $m(t)$, freeing up $m(t)$ to
be matched with $t$. So suppose this is blocked by some
$t''$ different from $t'$ (and necessarily different from $t$) with $m'\in\{m(t''),M(t'')\}$.  
But then $t'\mid tt''$, a violation of 2-primitivity.  Thus, the matching can be completed.
\end{proof}

\begin{lem} \label{lem3}
Suppose $0<\theta<1/3$ and that $T$ is $2$-primitive with $P(t) < t^{\theta}$ for each $t \in T$. Let $N(z)=|T\cap[2,z]|$. Then, with $q$ running over primes in the interval
$I:=[z^{(1+\theta)/4},z^{(1+\theta)/2})$, we have
\[
N(z) < z^{(1 + \theta)/2}-\sum_{q\in I}\Big\lfloor\frac{z^{(1+\theta)/2}}{q}\Big\rfloor
\]
\end{lem}

\begin{proof}
By Lemma \ref{lem1}, it suffices to bound $|{\mathcal M}(T\cap[2,z])|$.
  We first 
show that ${\mathcal M}(T\cap[2,z])\subset[1,z^{(1+\theta)/2})$.
Let $t\in T$ with $t\le z$.  Say $t = q_1 q_2 ... q_r$ where the primes $q_i$ are written in nondecreasing order.  Let $d = q_1 q_2 ... q_i$ be maximal with $d \le t^{(1 - \theta)/2}$.  Then
$d' = d q_{i+1}$ satisfies $t^{(1 - \theta)/2} < d' < t^{(1 + \theta)/2}$.
Also, $d'' = t/d'$ satisfies the same double inequality.  Thus, 
\[
t^{(1 - \theta)/2} < m(t) \le M(t) < t^{(1 + \theta)/2}\le z^{(1+\theta)/2}.
\]
We further note that the members $m$ of ${\mathcal M}(T\cap[2,z])$ satisfy $P(m)<z^\theta$, since $m$ divides
some member of $T\cap[2,z]$ and every $t$ in that set has $P(t)<z^\theta$.  In particular, $m$ is not divisible
by any prime $q\ge z^\theta$.  Note that if $\theta<1/3$, then $\theta <(1+\theta)/4$. So, $m$ is not divisible
by any prime in the interval $I$.  Since no integer below $z^{(1+\theta)/2}$
is divisible by 2 primes from $I$, the lemma follows.
\end{proof}
Set
\[
T^p=\{t\in T:P(t)=p\},
\]
so that $T^p\subset T_p$.  We have the following variant of Lemma \ref{lem3}.
\begin{lem}\label{lem:variant}
For any $2$-primitive set $T$ and prime $p$, let $N_p(z)$ denote the number of members $t$ of $T^p$ with
$t\le z$.  With $q$ running over the primes in $I_p:=(\max\{p,z^{1/4}\},z^{1/2})$, we have
\[
N_p(z)\le z^{1/2}-\sum_{q\in I_p}\Big\lfloor\frac{z^{1/2}}{q}\Big\rfloor.
\]
\end{lem}
\begin{proof}
Note that if $T$ is 2-primitive, so too is $T^p/p=\{t/p:t\in T^p\}$.  Thus, we may apply Lemma \ref{lem1}
to obtain a matching from $T^p/p$ into ${\mathcal M}(T^p/p)$.  The prime factors of each element $t/p\in T^p/p$
are at most $p$, so following the proof of Lemma \ref{lem3}, we have $m(t/p),M(t/p)\in[t^{1/2}/p,t^{1/2})$.
The lemma then follows in the same way as Lemma \ref{lem3}.
\end{proof}
\begin{lem}\label{lem:qharm}
For $x\ge2$ we have
\[
\sum_{\substack{x^{1/2}<q<x\\q\,{\rm prime}}}\Big\lfloor\frac xq\Big\rfloor \ge
\Big(\log 2-\frac{1.25}{\log x}-\frac{2.5}{(\log x)^2}\Big)x.
\]
\end{lem}
\begin{proof}
First suppose that $x\ge286^2$.  We have the sum is at least
\[
\sum_{x^{1/2}<q<x}\frac xq-\pi(x).
\]
 From \cite[(3.7)]{RS}, we have that $\pi(x)<1.25x/\log x$ and from \cite[(3.17)]{RS} that
\[
\sum_{q<x}\frac1q >\log\log x +B-\frac1{2(\log x)^2},
\]
where $B$ is the Mertens constant.  Further, from \cite[(3.18)]{RS},
\[
\sum_{q\le x^{1/2}}\frac1q<\log\log x^{1/2}+B+\frac1{2(\log x^{1/2})^2}=\log\log x-\log 2+B+\frac2{(\log x)^2}.
\]
This proves the lemma in the range $x\ge286^2$ and direct calculation shows that
it holds in the wider range $x\ge2$.
\end{proof}
We shall find it useful to use the following asymptotically weaker estimates
in small cases.  The proof follows by checking values of $x\le 3213$ after which
Lemma \ref{lem:qharm} is stronger.
\begin{cor}\label{cor:qharm}
For $x\ge185$, we have $\sum_{q\in(x^{1/2},x]}\lfloor x/q\rfloor >0.5x$.
For $x\ge 67$, we have $\sum_{q\in(x^{1/2},x]}\lfloor x/q\rfloor >0.45x$.
\end{cor}

\bigskip
Let
\begin{equation}
\label{eq:params}
\theta=0.3,\quad \lambda=0.7982562,\quad \nu=1/\theta=10/3.
\end{equation}
  For each prime $p$, let
\[
S_p=\sum_{\substack{t\in T\\P(t)=p<t^\theta}}\frac1{t^\lambda}.
\]
With \eqref{largeprime} it will suffice to prove Theorem \ref{thm5'} if we
show under its hypotheses that for each $Y\ge2$,
\begin{equation}
\label{eq:suff}
\sum_{p\le Y}S_p\le\sum_{p\le Y}\left(\frac1{p^\lambda}-\frac1{p^\tau}\right).
\end{equation}

\subsection{Small primes, $Y\le 37$}
We are going to estimate $S_p$ for various small primes $p$.
Take $t\in T$ with $P(t)<t^\theta$. If $t\le q^\nu$ for a prime $q$, then 
$P(t)<(q^\nu)^\theta=q$.
If $q=3$, we see there can be at most one such $t$; that is, $T$ can contain
at most one power of 2.  The values of $t\le 5^\nu$ are supported on $\{2,3\}$,
so by Lemma \ref{lem2} with $k=n=2$ we see that there are at most 2 such
members of $T$.  Similarly, Lemma \ref{lem2} with $k=2$, $n=3$ shows that
$T$ has at most 4 members below $7^\nu$, and with $k=2, n=4$, $T$ has at
most 19 members below $11^\nu$.
Since members $t$ of $T$ with $P(t)<t^\theta$ have at least $\lceil\nu\rceil=4$ prime factors (counted with multiplicity), we have
\begin{align}\label{smallY}
S_2&\le\frac1{2^{4\lambda}}<0.1093463,\notag\\
S_2+S_3&<0.1093463+\frac{2-1}{3^{\nu\lambda}}<0.1631052,\notag\\
S_2+S_3+S_5&<0.1631052+\frac{4-2}{5^{\nu\lambda}}<0.1907220,\notag\\
S_2+S_3+S_5+S_7&<0.1907220+\frac{19-4}{7^{\nu\lambda}}<0.2753295.
\end{align}

Computing $\sum_{p\le Y}(1/p^\lambda-1/p^\tau)$ directly for $Y=2,3,5,7$ gives lower bounds
\[
0.121399,~0.251741,~0.368904,~0.471733,
\]
respectively.  
Thus we observe $\sum_{p\le Y} S_p < \sum_{p\le Y}(1/p^\lambda-1/p^\tau)$, so by 
\eqref{eq:suff}, Theorem \ref{thm5'} holds 
when $Y=2,3,5,7$, respectively.  

Now consider $11\le p\le 37$. By partial summation, we have the equality
\begin{equation}
\label{eq:Sp}
S_p=\int_{p^\nu}^\infty\frac{\lambda}{z^{1+\lambda}}N_p(z)\,dz ,
\end{equation}
noting that the integral converges, since $N_p(z)\le z^{(1+\theta)/2}$ by Lemma \ref{lem3}.

We use Lemmas \ref{lem:variant} and \ref{lem:qharm} to get the upper estimates for $N_p(z)$:
\begin{align}
N_p(z)&\le \big\lfloor \sqrt{z}\big\rfloor-\sum_{\max(p,z^{1/4})<q\le \sqrt{z}}\Big\lfloor\frac{\sqrt{z}}q\Big\rfloor \label{eq:largez1},\\
N_p(z)&\le \sqrt{z}\Big(1-\log2+\frac{2.5}{\log z}+\frac{10}{(\log z)^2}\Big),\hbox{ when }p\le z^{1/4}.\label{eq:largez2}
\end{align}
We split the integral in \eqref{eq:Sp} at $p^4$.
In the first range when $z< p^4$, we bound the contribution to \eqref{eq:Sp} by splitting up into intervals $[m^2,(m+1)^2]$ and using \eqref{eq:largez1}, which gives
\begin{align}\label{eq:Sp'}
S_p':=\int_{p^\nu}^{p^4}\frac{\lambda}{z^{1+\lambda}}N_p(z)\,dz \
& \le \ \sum_{m_0<m<p^2}\int_{m^2}^{(m+1)^2}\frac{\lambda}{z^{1+\lambda}}N_p(z)\,dz \ + \
\int_{p^\nu}^{(m_0+1)^2}\frac{\lambda}{z^{1+\lambda}}N_p(z)\,dz \notag\\
&\le \ \sum_{m_0<m<p^2}\Big(\frac1{m^{2\lambda}}-\frac1{(m+1)^{2\lambda}}\Big)\Big(m-\sum_{p<q\le m}\big\lfloor\frac mq\big\rfloor\Big) \notag\\
&\qquad + \ \Big(\frac1{p^{\nu\lambda}}-\frac1{(m_0+1)^{2\lambda}}\Big)\Big(m_0-\sum_{p<q\le p^{\nu/2}}\big\lfloor\frac{m_0}{q}\big\rfloor\Big)
\end{align}
where $m_0 = \lfloor p^{\nu/2}\rfloor$.

For the second range when $z\ge p^4$, we use \eqref{eq:largez2} when $z\ge3213^2$ and for smaller values of $z$ we use Corollary \ref{cor:qharm}. That is,
\begin{align*}
S_p'' :=  \int_{p^4}^\infty\frac{\lambda}{z^{1+\lambda}}N_p(z)\,dz \ 
& \le \int_{3213^2}^\infty\frac{\lambda}{z^{\lambda+1/2}}\Big(1-\log2+\frac{2.5}{\log z}+\frac{10}{(\log z)^2}\Big)\,dz\\
& \qquad + \ 0.5\int_{\max(p^4, 185^2)}^{3213^2}\frac{\lambda}{z^{1/2+\lambda}}\,dz
 \ + \ 0.55\int_{p^4}^{\max(p^4, 185^2)}\frac{\lambda}{z^{1/2+\lambda}}\,dz.
\end{align*}
Denote the integrals
\begin{align*}
f(y) & := \int_y^\infty\frac{\lambda}{z^{\lambda+1/2}}\,dz \\
g(y) & := \int_y^\infty\frac{\lambda}{z^{\lambda+1/2}}\Big(1-\log2+\frac{2.5}{\log z}+\frac{10}{(\log z)^2}\Big)\,dz.
\end{align*}
So we obtain
\begin{align}\label{eq:Sp''}
S_p'' & \le \ (1-\log 2)f(3213^2) + g(3213^2) \notag\\
& \quad + \ 0.5[f(\max(p^4,185^2))-f(3213^2)] \ + \ 0.55[f(p^4)-f(\max(p^4,185^2))] \notag\\
& \ = \ (0.5-\log 2)f(3213^2) + g(3213^2) \ - \ 0.05f(\max(p^4,185^2)) \ + \ 0.55f(p^4).
\end{align}

Using the estimates in \eqref{eq:Sp'}, \eqref{eq:Sp''}, we bound $S_p = S_p' + S_p''$ by the following. 
\medskip

\begin{center}
\renewcommand{\arraystretch}{1.5}
    \begin{tabular}{|c| c| c |c|}
    \hline
    $p$&$S_p\le$& $ \sum_{q\le p}S_q\le$&~$\sum_{q\le p}(q^{-\lambda}-q^{-\tau})\ge$\\
    \hline
     11&0.13259&0.40792&0.55427\\
    13&0.11241&0.52033&0.62966\\
    17&0.08382&0.60415&0.69432\\
    19&0.07601&0.68016&0.75484\\
    23&0.06194&0.74210&0.80868\\
    29&0.04757&0.78967&0.85521\\
    31&0.04501&0.83468&0.89978\\
    37&0.03680&0.87148&0.93950\\
    \hline
    \end{tabular}
\end{center}
\medskip
\noindent
Note that the first entry in the third column is found by adding $S_{11}$ to the
estimate in \eqref{smallY}.    Since the entries in the fourth column
exceed the entries in the third column, \eqref{eq:suff} implies
Theorem \ref{thm5'} for $Y\le 37$.

\subsection{Large primes, $Y\ge 41$}
Now assume that $Y\ge41$.  We have via partial summation that
\begin{align*}
\sum_{\substack{t\in T\\P(t)<t^\theta}}\frac1{t^\lambda}=\sum_{p\le7}S_p
+\sum_{11\le p\le 23}\int_{p^\nu}^{29^\nu}\frac{\lambda}{z^{1+\lambda}}N_p(z)\,dz
+\int_{29^\nu}^\infty\frac{\lambda}{z^{1+\lambda}}N(z)\,dz.
\end{align*}
(As before, the last integral converges.)
 From \eqref{smallY} the $S_p$ terms contribute at most 0.27533. Using Lemmas \ref{lem3}, \ref{lem:variant}, and \ref{lem:qharm}, and Corollary \ref{cor:qharm},
we similarly obtain
 \begin{align*}
 \sum_{\substack{t\in T\\P(t)<t^\theta}}&\frac1{t^\lambda}\\
 &<0.27533+0.08455+0.06576
 +0.03756+0.02953+0.01487+0.45614=0.96374,
 \end{align*}
 where the second to the sixth terms correspond to the five finite integrals,
 and the last term is our estimate for the tail integral.  We also note that
 \[
  \sum_{p\le Y}\Big(\frac1{p^\lambda}-\frac1{p^\tau}\Big)\ge
 \sum_{p\le41}\Big(\frac1{p^\lambda}-\frac1{p^\tau}\Big)>0.97661.
 \]
Since this estimate exceeds the prior one, this gives Theorem \ref{thm5'}
with $\lambda=0.7982562$.

%%% AUTHOR: optional acknowledgments here
\section*{Acknowledgments} %%  you may comment this out if no Ackno
The authors would like to acknowledge the University of Memphis for hosting the 2019 Erd\H{o}s Lecture Series, during which the initial ideas for this paper were conceived.  The second author was supported by a Clarendon Scholarship at
the University of Oxford.
The third author is grateful for the
hospitality of Santa Clara University and was funded by their Paul R. and Virginia P.
Halmos Endowed Professorship in Mathematics and Computer Science.
The authors thank the referee for a careful reading.

%%% AUTHOR:
%%% Bibliography goes here. Note that the arXiv cannot process bibtex
%%% or biber bibliographies.  Example of acceptable bibliograpy format:
\bibliographystyle{amsplain}

%% AUTHOR: You can generate such a bibliography from a .bib file by 
%% running pdflatex/bibtex/pdflatex/pdflatex and then pasting the .bbl file
%% between \begin{thebibliography} and \end{bibliography}

%%% AUTHOR: Include a short description of each author following the
%%% structure below. Use the same short tags used previously.  
%%% Use \imageat{} and \imagedot{} instead of "@" and "." in
%%% email addresses-this replaces the symbols with graphics to avoid 
%%% e-mail address harvesting from the .pdf file
\begin{dajauthors}
\begin{authorinfo}[THC]
  Tsz Ho Chan\\
  Kennesaw State University\\
  Kennesaw, GA 30144\\
  thchan6174\imageat{}gmail\imagedot{}com \\
  \url{https://www.memphis.edu/msci/people/tchan.php}
\end{authorinfo}
\begin{authorinfo}[JDL]
  Jared Duker Lichtman\\
  University of Oxford\\
  Oxford, OX2 6GG, UK\\
  jared\imagedot{}d\imagedot{}lichtman\imageat{}gmail\imagedot{}com \\
  \url{https://www.maths.ox.ac.uk/people/jared.lichtman}
\end{authorinfo}
\begin{authorinfo}[CP]
  Carl Pomerance\\
  Dartmouth College\\
  Hanover, NH 03755\\
  carl\imagedot{}pomerance\imageat{}dartmouth\imagedot{}edu\\
  \url{https://math.dartmouth.edu/~carlp/}
\end{authorinfo}
\end{dajauthors}

\end{document}